\documentclass[12pt]{amsart}
\usepackage{fullpage}
\usepackage{amssymb,amsmath,amsthm,amsfonts,mathrsfs}
\usepackage{bm}
\usepackage{enumitem}
\usepackage{mathtools}
\usepackage{tensor}
\usepackage{tikz-cd}
\usepackage{hyperref}
\hypersetup{
  colorlinks=true,
  linkcolor={red!50!black},
  filecolor=magenta,
  urlcolor=cyan,
  citecolor={blue!80!black}
}
\usepackage{amsrefs}

\newcommand{\CC}{\mathbb{C}}

\newcommand{\RR}{\mathbb{R}}

\DeclareMathOperator{\dom}{dom}
\DeclareMathOperator{\Diff}{Diff}
\DeclareMathOperator{\Eff}{Eff}
\DeclareMathOperator{\germ}{germ}
\DeclareMathOperator{\id}{id}
\DeclareMathOperator{\Hol}{Hol}
\DeclareMathOperator{\loc}{loc}
\DeclareMathOperator{\pr}{pr}
\DeclareMathOperator{\spn}{span}
\DeclareMathOperator{\stab}{stab}

\newcommand{\ol}[1]{\overline{#1}}
\newcommand{\cplx}[2]{\Omega_{#1}^\bullet(#2)}

\newcommand{\bb}[1]{\mathbb{#1}}
\newcommand{\cl}[1]{\mathcal{#1}}
\newcommand{\sr}[1]{\mathscr{#1}}
\newcommand{\fk}[1]{\mathfrak{#1}}
\newcommand{\pdd}[2]{\frac{\partial #1}{\partial #2}}

\newcommand{\rra}{\rightrightarrows}

\numberwithin{equation}{section}

\newtheorem{theorem}                {Theorem}[section]
\newtheorem*{theorem*}              {Theorem}
\newtheorem{corollary}   [theorem]  {Corollary}
\newtheorem{lemma}       [theorem]  {Lemma}
\newtheorem{proposition} [theorem]  {Proposition}
\newtheorem*{property}              {Property}

\theoremstyle{definition}
\newtheorem{definition}  [theorem]  {Definition}

\theoremstyle{remark}
\newtheorem{remark}      [theorem]  {Remark}
\newtheorem{example}     [theorem]  {Example}
\newtheorem*{example*}              {Example}

\newenvironment{enumeratea}{
  \begin{enumerate}[label = (\alph*)]}{
  \end{enumerate}}
\newenvironment{enumerater}{
  \begin{enumerate}[label = (\roman*)]}{
  \end{enumerate}}

\newif\ifdebug

\debugtrue

\begin{document}

\begin{abstract}
  A singular foliation $\cl F$ gives a partition of a manifold $M$ into leaves whose dimension may vary. Associated to a singular foliation are two complexes, that of the diffeological differential forms on the leaf space $M / \cl F$, and that of the basic differential forms on $M$. We prove the pullback by the quotient map provides an isomorphism of these complexes in the following cases: when $\cl F$ is a regular foliation, when points in the leaves of the same dimension assemble into an embedded (more generally, diffeological) submanifold of $M$, and, as a special case of the latter, when $\cl F$ is induced by a linearizable Lie groupoid.
\end{abstract}
\title{The basic de Rham complex of a singular foliation}
\author{David Miyamoto}
\date{\today}
\maketitle

\section{Introduction}
\label{sec:introduction}
A singular foliation $\cl F$ of a manifold $M$ is a partition of $M$ into connected, weakly-embedded submanifolds, called leaves, of perhaps varying dimension, satisfying the following smoothness condition: for every $x \in M$ contained in a leaf $L_x$, there exist locally defined vector fields $X^i$ about $x$ such that
\begin{itemize}
\item the $X^i$ are tangent to the leaves, and
  \item the $X^i_x$ span $T_xL_x$
\end{itemize}
We then take the complex of basic differential forms to consist of those $\alpha \in \cplx{}{M}$ such that
\begin{equation*}
  \iota_X \alpha = 0 \text{ and } \cl L_X\alpha = 0, \text{ for all } X \text{ tangent to the leaves}.
\end{equation*}
  The set of $\cl F$-basic forms, denoted $\cplx{b}{M, \cl F}$, is a de Rham subcomplex of differential forms. We relate this complex to one on the quotient (or leaf) space $M/\cl F$. While rarely a smooth manifold - for example, the leaf space of the irrational flow on the torus is not even Hausdorff - $M/\cl F$ is naturally a diffeological space (introduced in Section \ref{sec:diffeology}).  As a diffeological space, $M/\cl F$ comes equipped with a de Rham complex of diffeological differential forms, $\cplx{}{M/\cl F}$, and the quotient $\pi:M \to M/\cl F$ is diffeologically smooth. Its pullback induces a one-to-one morphism from diffeological forms into basic forms. We seek singular foliations for which the following property holds
\begin{property}[P]
  The pullback $\pi^*:\cplx{}{M/\cl F} \to \cplx{b}{M, \cl F}$ is into (hence, an isomorphism).
\end{property}
Diffeology plays an important role. For example, we prove the foliation of the torus by an irrational flow has property (P), and as previously mentioned its leaf space is not a manifold. More generally, we prove
\begin{enumerate}[label = (\Alph*)]
\item \label{item:1} Regular foliations have property (P) (Theorem \ref{thm:5}).
\item \label{item:2} $\cl F$ has property (P) the union of leaves of the same dimension is an embedded (more generally, diffeological) submanifold (Theorem \ref{thm:6}).
\item \label{item:3} $\cl F$ has property (P) if the induced singular foliation on $M \smallsetminus \{x \mid \dim L_x = 0\}$ has property (P) (Theorem \ref{thm:7}).
\end{enumerate}
We proved \ref{item:1} independently, then found Hector, Marc\'{i}as-Virg\'{o}s, and Sanmart\'{i}n-Carb\'{o}n \cite{H2011} proved it earlier. We use groupoid techniques to approach this problem. For a groupoid version of property (P), replace $(M, \cl F)$ with a Lie groupoid $\cl G \rra M$; replace $M/\cl F$ with $M/\cl G$; and take the basic forms to be $\alpha$ such that $s^*\alpha = t^*\alpha$. Then, we arrive at \ref{item:1} by proving:
\begin{itemize}
\item property (P) is Morita-invariant,
\item the Holonomy groupoid has property (P) if and only if $\cl F$ does,
\item an \'{e}tale groupoid with countably generated pseudogroup (e.g. the \'{e}tale version of the Holonomy groupoid) has property (P).
\end{itemize}
Hector et al.\ deal directly with the action of Haefliger's holonomy pseudogroup on a complete transversal. Our methods constitute a generalization of Hector et al.'s approach, and we make explicit some correspondences left as routine in \cite{H2011}.

Item \ref{item:2} has not appeared in the literature, and extends a few existing results. In \cite{KW2016}, Karshon and Watts proved that given a Lie group $G$ acting on $M$ whose identity component $G^\circ$ acts properly, the corresponding action groupoid $G \ltimes M \rra M$ has property (P). We can alternatively get this as a consequence of \ref{item:2}. Palais \cite{P1961} showed that proper group actions admit slices, or equivalently, that their associated groupoids are linearizable, in the sense of being locally isomorphic to a linearized version of the action. In Appendix \ref{sec:line-strat-dimens}, Proposition \ref{prop:7}, we show that this implies the singular foliation of $M$ by the $G^\circ$ orbits satisfies the hypothesis in \ref{item:2}, hence has property (P). The fact $G^\circ$ is connected implies $G^\circ \ltimes M \rra M$ also has property (P) (Proposition \ref{prop:5}), and a lemma from \cite{KW2016} shows this suffices to conclude $G \ltimes M \rra M$ has property (P).

In \cite{W2015}, Watts extended \cite{KW2016} to prove that proper Lie groupoids have property (P). Assuming source-connectedness, we can give another argument. It is by now well-established that proper Lie groupoids are linearizable, in the sense of Definition \ref{def:22}. Weinstein made the first advances in this direction in \cite{W2002}, where he showed, under some extra conditions, that linearizability at fixed points was sufficient for linearizability at finite type orbits (a term we do not define here). Zung made the next significant advance in \cite{Z2006}, by proving proper Lie groupoids, again satisfying some extra assumptions, are linearizable at their fixed points. Crainic and Struchiner collected these various results in \cite{CS2013}, and themselves gave a self-contained proof that proper Lie groupoids are linearizable without qualification, which we state as Theorem \ref{thm:8}. By Proposition \ref{prop:7}, linearizability implies the hypotheses of \ref{item:2} hold, hence the foliation of $M$ by the orbits of $\cl G$ has property (P). By connectedness, this is equivalent to $\cl G$ having property (P).

A key feature of the above arguments is that linearizability of a (source-connected) Lie groupoid suffices to guarantee it has property (P). Since properness is sufficient but not necessary for linearizability, our class of examples for which (P) holds is distinct from those found in \cite{KW2016} and \cite{W2015}. However, even linearizability is not a necessary condition for (P). Our final result \ref{item:3} shows that the situation of the $0$-dimensional leaves do not impact whether $\cl F$ has property (P). For example, every singular foliation of $\bb R$ has property (P), and many of these are not linearizable.

Our paper is structured as follows. Sections \ref{sec:diffeology}, \ref{sec:singular-foliations}, and \ref{sec:lie-algebr-group} review diffeology, singular foliations, and Lie groupoids, respectively. In Section \ref{sec:basic-de-rham}, we state and prove our main results. Appendix \ref{sec:line-strat-dimens} is a review of linearizability, and contains the proof that linearizable Lie groupoids satisfy the hypotheses in \ref{item:2}. Appendix \ref{sec:pseudogroups} reviews pseudogroups.

\subsection*{Acknowledgements}
\label{sec:acknowledgements}

I am grateful to Yael Karshon for introducing this question to me, and providing constant and positive support throughout the writing process. I also thank Jordan Watts for directing me to sources on Lie groupoids, and Camille Laurent-Gengoux for helpful comments on the first version. This research is partially supported by the Natural Sciences and Engineering Research Council of Canada.

\section{Diffeology}
\label{sec:diffeology}

We use diffeology to handle singular spaces, so we review of its basic concepts. Our reference is Iglesias-Zemmour's book \cite{I2013}.

\subsection{Diffeology and Manifolds}
\label{sec:first-notions}

\begin{definition}[Diffeology]
\label{def:1}
  Let $X$ be a set. A \emph{parametrization} into $X$ is a map from an open subset of a Cartesian space into $X$.  A \emph{diffeology} on $X$ is a set $\sr D$ of parametrizations, whose members are called \emph{plots}, such that

  \begin{itemize}
  \item Constant maps are plots.
  \item If a parametrization $P:U \to X$ is such that about each $r \in U$, there is an open $V \subseteq U$ and a plot $Q:V \to X$ such that $P=Q|_V$, then $P$ is a plot.
  \item If $P:U \to X$ is a plot and $V$ is an open subset of a Cartesian space, then for any smooth $F:V \to U$, the pre-composition $F^* P$ is a plot.
  \end{itemize}

  A space equipped with a diffeology is a \emph{diffeological space}.
\end{definition}

The set of locally constant parametrizations into $X$, and the set of all parametrizations into $X$, are both diffeologies, called the respectively \emph{discrete} and \emph{coarse}. Every other diffeology sits between these two. A classical smooth manifold\footnote{a topological space locally homeomorphic to $\RR^n$ for some $n$, equipped with a compatible smooth atlas. We do not assume Hausdorff nor second-countable here.} $M$ carries a canonical diffeology $\sr D_M$, consisting of the smooth maps (in the usual sense) from Cartesian spaces into $M$.

\begin{definition}[Smooth maps]
\label{def:2}  
We say a map $f:X \to Y$ between diffeological spaces is \emph{diffeologically smooth} if for every plot $P$ of $X$, the pre-composition $P^*f$ is a plot of $Y$. Denote the smooth maps from $X$ to $Y$ by $C^\infty(X, Y)$. If $Y = \RR$, we write $C^\infty(X) := C^\infty(X, \RR)$.
\end{definition}

If $(X, \cl D_X)$ is discrete or $(Y, \cl D_Y)$ is coarse, all maps $X \to Y$ are smooth. A map between classical manifolds is diffeologically smooth if and only if it is smooth in the usual sense.

\begin{definition}[D-topology]
  \label{def:3}
  The D-topology on a diffeological space $X$ is the finest topology in which all plots are continuous. In other words, $A \subseteq X$ is D-open if and only if $P^{-1}(A)$ is open for all plots $P$.
\end{definition}

Every diffeologically smooth map is continuous with respect to the D-topology. For a classical manifold $M$, the D-topology associated to $\sr D_M$ is the original topology on $M$. The D-topology lets us define local objects in diffeology, in particular locally smooth maps, local diffeomorphisms, and diffeological manifolds.

\begin{definition}
  \label{def:4}\
  \begin{itemize}
  \item A partially defined function $f:A \subseteq X \to X'$ is \emph{locally smooth} if $A$ is D-open, and for every for every plot $P$ of $X$, the pre-composition $P^*f:P^{-1}(A) \to X'$ is a plot. Denote the set of locally smooth maps by $C_{\loc}^\infty(X, X')$.
  \item If $f:A \subseteq X \to X'$ is locally smooth, injective, and $f^{-1}:f(A) \subseteq X' \to X$ is locally smooth, we call $f$ a \emph{local diffeomorphism}. Denote the set of local diffeomorphisms by $\Diff_{\loc}(X, X')$.
    \item For fixed $n$, if about every $x \in X$ there is a local diffeomorphism $f:A \subseteq X \to \RR^n$, we call $(X, \sr D)$ a \emph{diffeological manifold} of dimension $n$.
  \end{itemize} 
\end{definition}

The charts of a classically smooth $n$-manifold $M$ are local diffeological diffeomorphisms into $\RR^n$. Therefore $(M, \sr D_M)$ is a diffeological $n$-manifold. Conversely, the set of local diffeological diffeomorphisms from a diffeological $n$-manifold $X$ into $\RR^n$ is a maximal atlas for a smooth structure $\sigma$ on $X$ that is compatible with the D-topology. Thus $(X, \sigma)$ is a classical (perhaps not Hausdorff nor second-countable) $n$-manifold.  This correspondence describes a full, faithful functor from the category of classical manifolds to that of diffeological spaces, whose image is the diffeological manifolds.

Beyond classical manifolds, the following three spaces inherit a natural diffeology: subsets of a diffeological space, quotients of a diffeological space, and function spaces of diffeologically smooth maps.

\begin{definition}[Subset diffeology]
  \label{def:5}\
  \begin{itemize}
  \item For a subset $S$ of a diffeological space $X$, with inclusion $\iota: S\hookrightarrow X$, the \emph{subset diffeology} on $S$ consists of all parametrizations $P:U \to S$ such that $\iota \circ P$ is a plot of $X$. The D-topology of the subset diffeology is not necessarily the subset topology.
    \item A subset $S$ is a \emph{diffeological submanifold} of $X$ if it is a diffeological manifold when equipped with the subset diffeology.
\end{itemize}
\end{definition}

We take a moment to relate diffeological submanifolds to the more common notion of weakly-embedded submanifolds:

\begin{definition}
  \label{def:6}
  A subset $S$ of a classical smooth manifold $M$ is a \emph{weakly-embedded submanifold} if it is a diffeological submanifold of $M$, and the inclusion $\iota:S \hookrightarrow M$ is an immersion.
\end{definition}

\begin{remark}
  \label{rem:1}
  If $S$ is weakly-embedded, it admits a unique topology and smooth structure of a classical smooth manifold. Therefore, when a subset is weakly-embedded, we may speak of it as a (sub)manifold without ambiguity.
\end{remark}

A natural question arises: is every diffeological submanifold of a classical smooth manifold also weakly-embedded? Perhaps surprisingly, the answer is no. Here is a counter-example, proposed by Y. Karshon and J. Watts.
Take $M = \RR^2$, and set
\begin{equation*}
  S := \{(x,y) \in \RR^2 \mid x^2 = y^3\}.
\end{equation*}

\begin{proposition}
  \label{prop:1}
  The map $f:\RR \to S$ given by $f(t) = (t^3, t^2)$ is a diffeological diffeomorphism, where $S$ carries the subset diffeology. Hence $S$ is a diffeological submanifold.
\end{proposition}

Observing that the inclusion is not an immersion at $(0,0)$ completes the example. The proof of Proposition \ref{prop:1} invokes a 1982 result of H. Joris \cite{J1982}, which states that for every parametrization $Q:V \to \RR$, the composition $f \circ Q$ is smooth (as a map to $\RR^2$) if and only if $Q$ is smooth. Joris' theorem also answers an open question in diffeology posed by Iglesias-Zemmour. Details will appear in an upcoming paper written jointly with Y. Karshon and J. Watts.

\begin{definition}
  \label{def:7}
 Given a diffeological space $X$ and a relation $\cl R$ with quotient map $\pi:X \to X/\cl R$, the \emph{quotient diffeology} on $X/\cl R$ consists of those parametrizations $P:U \to X/\cl R$ such that about each $r \in U$, there is an open $V\subseteq U$ and a plot $Q:V \to X$ of $X$ such that $P|_V = \pi \circ Q$. In a diagram,
       \begin{equation*}
    \begin{tikzcd}
       &  & X \ar[d, "\pi"] \\
      r \in V \ar[r, hook] \ar[urr, "\exists Q"] & U \ar[r, "P"] &  X/\cl R
    \end{tikzcd}
  \end{equation*}
 The D-topology of the quotient diffeology always coincides with the quotient topology..
\end{definition}
 
\begin{definition}[Functional diffeology]
  \label{def:8}
  Let $X$ and $Y$ be diffeological spaces. The \emph{standard functional diffeology} on $C_{\loc}^\infty(X, Y)$ consists of those parametrizations $P:U \to C_{\loc}^\infty(X, Y)$ satisfying: about each $r_0 \in U$ and $x_0 \in \dom P(r_0)$, there are open neighbourhoods $V \subseteq U$, and (D-open) $A \subseteq \dom P(r_0)$ such that
  \begin{itemize}
  \item $A \subseteq \dom P(r)$ for all $r \in V$, and
  \item the map $V \times A \to Y$ given by $(r, x) \mapsto P(r)(x)$ is smooth.
  \end{itemize}
\end{definition}

We equip $C^\infty(X, Y)$ with the subset diffeology induced from $C_{\loc}^\infty(X, Y)$. This is the coarsest diffeology in which the evaluation map $X \times C^\infty(X, Y) \to Y$ given by $(x, f) \mapsto f(x)$ is smooth. We similarly equip $\Diff_{\loc}(X, Y)$ and $\Diff(X, Y)$ with their subset diffeologies. In the global case, the composition map is diffeologically smooth. For a classical manifold $M$, we can prove the inversion map on $\Diff(M)$ is also smooth. 

\begin{proposition}
  \label{prop:2}
  The map $\Diff(M) \to \Diff(M)$ given by $g \mapsto g^{-1}$ is smooth.
\end{proposition}
\begin{proof}
   Let $P:U \to \Diff(M)$ be a plot, meaning $P:U \to C^\infty(M, M)$ is a plot with image in $\Diff(M)$.  It suffices to show $r \mapsto P(r)^{-1}$ is smooth as a map into $C^\infty(M, M)$, so we prove
  \begin{equation}\label{eq:1}
    U \times M \to M, \quad (r, x) \mapsto P(r)^{-1}(x)
  \end{equation}
  is smooth. Let $\cl P:U \times M \to U \times M$ be the map $(r, x) \mapsto (r, P(r)(x))$, and similarly define $\cl P^{-1}$. Since $P$ is a plot, $\cl P$ is smooth. Furthermore, it is a submersion: fixing $(r_0, x_0) \in U \times M$ and $y_0 := P(r_0)(x_0)$, the map $M \to U \times M$ given by $y \mapsto (r_0, P(r_0)^{-1}(y))$ is a smooth (since $r_0$ is fixed) section of $\cl P$ through $(r_0, x_0)$. Therefore by the inverse function theorem, locally $\cl P$ admits a smooth inverse. But this inverse is exactly $\cl P^{-1}$, so this map, and hence \eqref{eq:1}, is smooth.
\end{proof}

Because composition and inversion are smooth, we call $\Diff(M)$ a \emph{diffeological group}.  Any subgroup of $\Diff(M)$ with the subset diffeology is also a diffeological group. It is an open question whether $\Diff(X)$ is a diffeological group for arbitrary diffeological spaces..

\subsection{Diffeological Forms}
\label{sec:diffeological-forms}

Now we introduce diffeological differential forms, and state two useful results for quotients.

\begin{definition}[Diffeological forms]\label{def:9}
  A \emph{diffeological k-form} $\alpha$ on $X$ is an assignment to each plot $P:U \to X$ a differential $k$-form $\alpha(P) \in \Omega^k(U)$ such that for every open subset $V$ of a Cartesian space, and every smooth map $F:V \to U$, we have
  \begin{equation*}
    \alpha(P \circ F) = F^*(\alpha(P)).
  \end{equation*}
Denote the set of diffeological $k$-forms by $\Omega^k(X)$, and the set of diffeological forms by $\cplx{}{X}$.
\end{definition}

As with usual differential forms, diffeologial forms pull back under smooth functions.

\begin{definition}\label{def:10}
  Let $f :X \to Y$ be smooth, and $\alpha \in \Omega^k(Y)$.  The \emph{pullback} $f^*\alpha \in \Omega^k(X)$ is defined on the plots of $X$ by $(f^*\alpha)(P) := \alpha(f \circ P)$.
\end{definition}

The set $\Omega^k(X)$ is naturally a real vector space: given $\alpha, \beta \in \Omega^k(X)$ and $\lambda \in \RR$, define for every plot $P:U \to X$,
\begin{equation*}
  (\lambda\alpha + \beta)(P) := \lambda \alpha(P) + \beta(P).
\end{equation*}
The space $\Omega^\bullet(X)$ also carries a differential $d$ and wedge product $\wedge$, respectively defined by 
\begin{equation*}
  (d\alpha)(P) := d\alpha(P), \quad (\alpha \wedge \beta)(P) := \alpha(P) \wedge \beta(P).
\end{equation*}
With respect to the grading $\cplx{}{X} = \bigoplus_{k=0}^\infty \Omega^k(X)$, the space $(\cplx{}{X}, d, \wedge)$ is a differential commutative graded algebra.  The pullback by a smooth function is a morphism of commutative differential graded algebras.

Consider a classical manifold $M$.  To each usual differential form $\underline \alpha$ we may associate a diffeological form $\alpha$ by $\alpha(P) := P^*\underline \alpha$. Conversely, to each diffeological form $\alpha$,  we can specify a usual form $\underline \alpha$ by declaring that in each chart $\varphi$ of $M$, we have $\varphi^*\underline \alpha := \alpha(\varphi)$. By identifying $\alpha$ and $\underline \alpha$, we identify the usual de Rham complex on $M$ with the diffeological one. Under this identification, the pullback by a smooth function viewed in the classical and diffeological senses agree. Therefore, we may freely switch between viewing classical forms on $M$ as diffeological ones, and vice versa. 

The next two results appear in \cite{I2013} as Articles 6.38 and 6.39, respectively.

\begin{proposition}
  \label{prop:3}
  Fix a diffeological space $X$ with relation $\cl R$, and equip $X/\cl R$ with the quotient diffeology. The image of $\pi^*:\Omega^k(X/\cl R) \to \Omega^k(X)$ is those $k$-forms $\alpha$ on $X$ such that for any two plots $P, Q:U \to X$ with $\pi \circ P = \pi \circ Q$, we have
  \begin{equation*}
    \alpha(P) = \alpha(Q).
  \end{equation*}
\end{proposition}

\begin{remark}\label{rem:2}
  Proposition \ref{prop:3} is equivalent to the following statement.  Consider the diagram

  \begin{equation*}
    \begin{tikzcd}
      X \times_\pi X \ar[r, "\pr_1"] \ar[d, "\pr_2"] & X \ar[d, "\pi"] \\
      X \ar[r, "\pi"] & X/\cl R
    \end{tikzcd}
  \end{equation*}

  Then $\alpha$ is in the image of $\pi^*$ if and only if $\pr_1^*\alpha = \pr_2^*\alpha$. Equivalently, $\alpha$ is basic with respect to the diffeological relation groupoid $X \tensor[]{\times}{_\pi} X \rra X$ (c.f.\ Section \ref{sec:basic-de-rham}).
\end{remark}

We omit the proof.

\begin{lemma}\label{lem:1}
The pullback $\pi^*:\Omega^k(X/\cl R) \to \Omega^k(X)$ is injective. 
\end{lemma}

\begin{proof}
  The pullback $\pi^*$ is a linear map, so it suffices to prove its kernel is trivial.  Suppose $\beta \in \Omega^k(X/\cl R)$ satisfies $\pi^*\beta = 0$.  To show $\beta = 0$, we must show $\beta(P) = 0$ for any plot $P:U \to X/\cl R$.  We may do this locally. Let $r \in U$, and by definition of the pushforward diffeology, take a plot $Q:V \to X$ locally lifting $P$ about $r$.  Then
  \begin{equation*}
    \beta(P)|_V = \beta(P|_V) = \beta(\pi \circ Q) = (\pi^*\beta)(Q) = 0.
  \end{equation*}
So $\beta(P)$ vanishes locally, hence on $U$.
\end{proof}

\section{Singular Foliations}
\label{sec:singular-foliations}

We take the view that a singular foliation of a manifold is a partition of the manifold into connected immersed submanifolds assembling together smoothly in some sense. To develop this sense, we use the notion of a smooth singular distribution on a manifold. This is a slightly different, but equivalent, approach to that taken in the introduction.

Fix a classical manifold $M$, here always assumed to be Hausdorff and second-countable. We will drop the word ``classical'' for this section. Let $\cl V(M)$ denote the collection of vector fields of $M$ defined over open domains, and let $\fk X(M)$ denote the collection of globally defined vector fields. 

\begin{definition}
  \label{def:11}\   
  \begin{itemize}
\item A \emph{rough singular distribution} $\Delta$ over $M$ is an assignment to each $x \in M$ a linear subspace $\Delta_x$ of the tangent space $T_xM$.
  \item A \emph{section} of $\Delta$ is any smooth vector field $X \in \cl V(M)$ such that $X_y \in \Delta_y$ for every $y \in \dom X$. Denote the set of sections of $\Delta$ by $\Gamma_{\text{loc}}(\Delta)$.
    \item If for every $x \in M$ and $v \in \Delta_x$, there is a section of $\Delta$ through $v$, i.e.\ some $X \in \Gamma_{\text{loc}}(\Delta)$ such that $X_x = v$, then we say $\Delta$ is \emph{smooth}. 
  \end{itemize}
\end{definition}

\begin{remark}
  \label{rem:3}
  An equivalent condition for smoothness is: about each $x \in M$, there are vector fields $X^i \in \cl V(M)$ defined near $x$ such that $\spn(X^i_x) = \Delta_x$ and for all $i$ and all $y \in \dom(X^i)$, we have $X^i_y \in \Delta_y$. This is the analogue of the definition of smooth singular foliation presented in the introduction.
\end{remark}

An relevant question is: given a smooth singular distribution, about each $x \in M$, do there exist vector fields $X^i \in \cl V(M)$ defined on $U \ni X$ such that $\spn(X^i_y) = \Delta_y$ for all $y \in U$? This was resolved in the affirmative independently by Sussmann \cite{Su2008}, and by Drager, Lee, Park, and Richardson (DLPR) \cite{DLPR2012}. We will use this fact later, so state it here.

\begin{theorem}[Sussman-DLPR]
  \label{thm:1}
  Given a smooth singular distribution $\Delta$, there always exists $X^i \in \fk X(M)$ such that $\spn(X^i_x) = \Delta_x$ for all $x \in M$. Moreover, the $X^i$ can be chosen to be complete.
\end{theorem}

 For a smooth singular distribution, consider the dimension function
 \begin{equation*}
M \to \mathbb N, \quad x \mapsto \dim \Delta_x.
 \end{equation*}
 This lets us classify points as regular or singular: points $x \in M$ for which there is a neighbourhood $U$ on which $x \mapsto \dim \Delta_x$ is constant are the \emph{regular points}, and all others are \emph{singular}. We also introduce the subsets $M_{*k}$ of $M$, for $* \in \{=,\geq,>,<,\leq,\neq\}$, defined by
 \begin{equation*}
   M_{*k} := \{x \in M \mid \dim \Delta_x * k\}.
 \end{equation*}
The dimension map is lower semi-continuous, so the sets $M_{\geq k}$ are open. In particular, one can show the set of regular points is open and dense in $M$.

\begin{definition}
  \label{def:12}
    A \emph{singular foliation} of $M$ is a partition $\cl F$ of $M$ such that its members, called \emph{leaves}, are weakly-embedded and connected, and the associated singular distribution $(\Delta_{\cl F})_x := T_xL$ is smooth. Here $L$ is the leaf of $\cl F$ containing $x$.
  \end{definition}

  \begin{remark}
    \label{rem:4}
    For weakly-embedded submanifolds, see Definition \ref{def:6}. Since the weakly-embedded submanifolds admit unique smooth structures, a singular foliation does indeed depend only on the partition of $M$, and not on any choice of structure on its leaves. Also, because we assume $M$ is Hausdorff and second-countable here, it follows that the leaves are Hausdorff and second-countable, too.
  \end{remark}
  
We call a leaf \emph{regular} or \emph{singular} according to whether its points are regular or singular with respect to $\Delta_{\cl F}$.  The union of regular leaves is open and dense in $M$.

  Stefan in \cite{S1973} and Sussmann in \cite{Su1973} independently made groundbreaking progress in the study of singular foliations. Here we briefly review their results. Stefan originally worked with what he called ``arrows'' on a manifold $M$. While his terminology was not widely adopted, these are natural objects in diffeology.

  \begin{definition}
  \label{def:13}
  An \emph{arrow} on a manifold $M$ is a plot $a:U \subseteq \RR \to \Diff_{\loc}(M)$ such that $a(0) = \id$ and if $x \in \dom a(t)$, then $x \in \dom a(s)$ for every $0 \leq s \leq t$. Note $a(t)$ might be the empty map.
\end{definition}

If $X \in \cl V(M)$ has flow $\Psi$, the map $t \mapsto \Psi(t, \cdot)$ is an arrow. However, in general we do not require arrows to satisfy a group law. Take a collection of arrows $\cl A$ on $M$, and set

\begin{itemize}
\item $\Psi \cl A$ to be the pseudogroup (see Appendix \ref{sec:pseudogroups}) generated by $\bigcup_{a \in \cl A} a(\RR)$.
\item $(\Delta_{\cl A})_x := \left\{\frac{d}{dt}\Big|_{t = t_0} a(t,y) \mid a \in \cl A, \ a(t_0,y) = x\right\}$.
  \item $(\overline{\Delta_{\cl A}})_x := \{d\varphi_{\varphi^{-1}(x)}(v) \mid \varphi \in \Psi \cl A, \ v \in (\Delta_{\cl A})_{\varphi^{-1}(x)}\}$.
  \end{itemize}

  One can verify that both $\Delta_{\cl A}$ and $\overline{\Delta}_{\cl A}$ are smooth singular distributions. The following appears as Theorem 1 in \cite{S1973}.
\begin{theorem}[Stefan]
  \label{thm:2}
  Let $\cl A$ be a collection of arrows on a manifold $M$. The orbits of $\Psi \cl A$ are leaves of a singular foliation of $M$. For each leaf, $T_xL = (\overline{\Delta_{\cl A}})_x$. 
\end{theorem}

Now we discuss Sussmann's contribution. This next result can also be derived from Stefan's theorem after a bit of work, and is therefore called the Stefan-Sussmann theorem. We first require the notion of an integrable smooth singular distribution.

\begin{definition}
  \label{def:14}
   Fix a smooth singular distribution $\Delta$.
  \begin{itemize}
  \item An \emph{integral submanifold} of $\Delta$ through $x \in M$ is an immersed submanifold $L$ containing $x$ such that for every $y \in L$, we have $T_yL = \Delta_y$.
  \item An integral submanifold $L$ is \emph{maximal} if it is connected, and every other connected integral submanifold intersecting $L$ is an open submanifold of $L$. If a maximal integral submanifold exists about $x$, it is unique.
    \item We say $\Delta$ is \emph{integrable} if it admits an integral submanifold at each point in $M$.
  \end{itemize}
\end{definition}
  
It is immediate from the definition of a singular foliation that the associated smooth singular distribution $\Delta_{\cl F}$ is integrable. The Stefan-Sussmann theorem essentially states that all integrable singular distributions are $\Delta_{\cl F}$ for some singular foliation $\cl F$.

\begin{theorem}[Stefan-Sussmann]
  \label{thm:3}
  Fix a smooth singular distribution $\Delta$. The following are equivalent.

  \begin{enumeratea}
  \item $\Delta$ is integrable.
  \item There is a maximal integral submanifold about each point of $M$.
    \item  Let $\cl D$ be a ``spanning set'' of locally-defined vector fields, meaning $\Delta_x = \spn \{X_x \mid X \in \cl D\}$, and take $\cl A$ to be those arrows of the form $\Psi(t, \cdot)$, where $\Psi$ is the flow of some $X \in \cl D$. Then $\Delta = \overline{\Delta_{\cl A}}$.
  \end{enumeratea}
  In the case of (a), (b), or (c), the orbits of $\Psi \cl A$ are exactly the maximal integral submanifolds of $\Delta$. By Stefan's theorem, these are weakly-embedded, and therefore form a singular foliation $\cl F$ of $M$ such that $\Delta = \Delta_{\cl F}$.
\end{theorem}

Sussmann called the condition $\Delta = \overline{\Delta_{\cl A}}$ ``$\cl D$-invariance'' of $\Delta$. Stefan called it ``homogeneity'' of $\cl A$.  Historically, it was only Stefan who noted the maximal integral submanifolds of $\Delta$ are not just immersed, but weakly-embedded. It was only Sussmann who noted the maximality property of these orbits.

As a corollary to the Stefan-Sussmann theorem, we can show that every singular foliation is uniquely determined by its associated integrable singular distribution. This generalizes the Frobenius theorem for regular foliations, in the following sense: for a regular foliation $\cl F$, the Frobenius implies that correspondence $\cl F \mapsto \Delta_{\cl F}$ between foliations of $M$ and integrable regular distributions is a bijection.

\begin{corollary}\label{cor:1}
  The map
  \begin{equation*}
    \begin{split}
        \{\text{singular foliations}\} &\to \{\text{integrable singular distributions}\}\\
  \cl F &\mapsto \Delta_{\cl F}
    \end{split}
  \end{equation*}
  is a bijection.
\end{corollary}

\begin{proof}
Smoothness of $\Delta_{\cl F}$ comes from the definition of singular foliation, and the leaves $\cl F$ witness integrability. Now for integrable $\Delta$, define $\cl F_\Delta$ to be the partition of $M$ into its maximal integral submanifolds. This is a well-defined singular foliation by the Stefan-Sussmann theorem. It is now straightforward to check $\Delta \mapsto \cl F_\Delta$ and $\cl F \mapsto \Delta_{\cl F}$ are inverses.
\end{proof}

Let us finally give some examples of singular foliations.

\begin{example}
  Fix any smooth non-negative bounded function $f: \RR \to \RR$, and let $C := f^{-1}(0)$. Consider the smooth singular distribution $\Delta$ on $\RR$ spanned by $\cl D := \{f\pdd{}{x}\}$. The points of $C$ and the connected components of $\RR - C$ constitute maximal integral submanifolds of $\Delta$, hence form a singular foliation of $\RR$.  The set $M_{=0}$ is just $C$, and $M_{=1}$ is $\RR - C$.  As expected, $M_{=1}$ is open, and the regular points are $\text{int}(C) \cup M_{=1}$, which is dense in $\RR$.

  In fact, every closed subset of $\RR$ may be realized as $C$ above. Therefore the collection of singular leaves of a singular foliation may be quite wild. For instance, there is a singular foliation of $\RR$ whose singular points form a Cantor set.

  By a similar argument, the partition of $M$ into maximal integal curves of any complete vector field $X \in \fk X(M)$ is a singular folaition of $M$.
\end{example}

\begin{example}
  \label{ex:1}
  Let $\cl G \rra M$ be a Lie groupoid over $M$, with associated Lie algebroid $(\rho, \fk g)$ (see the next section for Lie groupoids and algebroids). Take $\cl F$ to be the partition of $M$ into the connected components of the orbits of $\cl G$, each equipped with their canonical smooth structure.\footnote{that is, the structure induced by viewing $t:s^{-1}(x) \to \cl O_x$ as a principal $G_x$-bundle} Then the distribution associated to $\cl F$ is exactly $\rho(\fk g)$, and one can show this is smooth. Therefore by definition, $\cl F$ is a singular foliation.

  Since the connected component of the identity $\cl G^\circ \rra M$ has the same Lie algebroid as $\cl G$, Corollary \ref{cor:1} implies the singular foliation induced by $\cl G^\circ$ is equal to $\cl F$. But the leaves of the former are exactly the orbits of $\cl G^\circ$, since these are already connected. Therefore the connected components of the $\cl G$ orbits are the $\cl G^\circ$ orbits.

  We can apply this discussion to obtain the same statements for a Lie group acting smoothly on $M$, by considering the action groupoid. If we consider the $\RR$ action on $M$ induced by the flow of a complete vector field, we recover the first example.
\end{example}

\begin{example}
  \label{ex:2}
  Given an arbitrary Lie algebroid $(\rho, \cl A)$, one can show the distribution $\rho(\cl A)$ is integrable using the Stefan-Sussmann theorem. But historically, Hermann in 1962 \cite{He1962} gave a sufficient condition for integability that predates Stefan and Sussmann's work, is easier to check, and is satisfied by $\rho(\cl A)$. This example subsumes the previous two. It is unknown whether every integrable distribution is of the form $\rho(\cl A)$ for some Lie algebroid $\cl A$.
\end{example}

\section{Lie Algebroids and Groupoids}
\label{sec:lie-algebr-group}

\subsection{Bibundles and Morita Equivalence}
\label{sec:morita-equiv-thro}

Our reference for this review is \cite{Ler2010}. A \emph{Lie groupoid} is a small category $\cl G \rra G_0$ with invertible arrows, such that the base $G_0$ is a Hausdorff, second-countable smooth manifold, the arrow space $\cl G$ is a smooth manifold (not necessarily Hausdorff or second-countable), all structure maps are smooth, and furthermore the source map is a smooth submersion with Hausdorff fibers.

We denote the source-fiber at $x$ by $P_x := s^{-1}(x)$ and the isotropy group at $x$ by $G_x := s^{-1}(x) \cap t^{-1}(x)$.  Because $s$ is a submersion, $P_x$ is an embedded submanifold of $\cl G$. It can also be shown that $G_x$ is a Lie group. The \emph{orbit} of $\cl G$ through $x \in G_0$ is $\cl O := t(P_x)$. This is an equivalence class of the relation $x \sim y$ if there is an arrow $x \mapsto y$. We can equip $\cl O$ with a canonical smooth structure such that $t:P_x \to \cl O$ is a principal $G_x$-bundle (see Theorem 5.4 in \cite{MM2003}), and $\cl O$ is immersed in $G_0$. Denote the orbit space by $M/\cl G$.

\begin{example}
  \label{ex:3}
Our first example of a Lie groupoid is the \emph{action groupoid} associated to a Lie group $G$ acting smoothly on a manifold $M$. The base is $M$, and the arrow space is $G \times M$.  The source and target maps are
\begin{equation*}
  s(g,x) = x, \quad t(g,x) = g \cdot x.
\end{equation*}
The unit map is $u(x) = (e,x)$, and the inversion is $(g,x)^{-1} = (g^{-1},x)$.  The multiplication is given by $(h,y)(g,x) = (hg,x)$, and is defined whenever $g \cdot x = y$. We denote the action groupoid by $G \ltimes M$. Its orbits are exactly the orbits of $G$.  
\end{example}

Just as a Lie group can act on a manifold, we also define a (right) Lie groupoid action on a manifold.

\begin{definition}
  \label{def:15}
  A \emph{right action} of a Lie groupoid $\cl H$ on a manifold $P$ consists of an anchor map $a_R:P \to H_0$ and a multiplication $P \tensor[_{a_R}]{\times}{_t} \cl H \to P$ fitting into the diagram
  \begin{equation*}
        \begin{tikzcd}
      & \ar[dl, shift left, outer sep=3, "\mu"'] \ar[dl, shift right, outer sep=3, "\pr_1" ] P \tensor[_{a_R}]{\times}{_t} \cl H \ar[dr, "\pr_2"] & \\
      P \ar[dr,"a_R"] & & \cl H \ar[dl, shift left, outer sep=3, "s"'] \ar[dl, shift right, outer sep=3, "t"] \\
      & H_0 &
    \end{tikzcd}
  \end{equation*}
  Moreover $(p\cdot h) \cdot h' = p \cdot (hh')$ whenever this makes sense, and $p \cdot 1|_{a(p)} = p$ for all $p \in P$. 
\end{definition}

We can view $P \tensor[_{a_R}]{\times}{_t} \cl H \rra P$ as a Lie groupoid itself, denoted $P \rtimes \cl H$, with source map $\mu$ and target $\pr_1$. The multiplication is $(p, h)(p', h') = (p, hh')$. Note $(a_R, \pr_2)$ is a morphism of Lie groupoids $P \rtimes H \to \cl H$.

\begin{definition}
  \label{def:16}
  A \emph{principal right} $\cl H$ \emph{bundle} is a surjective submersion $\pi:P \to B$ and a right $\cl H$ action on $P$ such that $\pi(p \cdot h) = \pi(p)$ whenever this makes sense, and $\cl H$ acts freely and transitively on the fibers of $\pi$. Precisely, the map $P \tensor[_{a_R}]{\times}{_t} \cl H \to P \tensor[]{\times}{_B} P$ given by $(p, h) \mapsto (p, p\cdot h)$ is a diffeomorphism.
\end{definition}

\begin{remark}
  \label{rem:5}
  We can view the bundle $P \tensor[]{\times}{_B} P \to P$ as a Lie groupoid $P \tensor[]{\times}{_B}P \rra P$, where the first and second projections constitute the source and target maps. When we have a principal right $\cl H$ bundle, we get a morphism of Lie groupoids $P \rtimes \cl H \to P \tensor[]{\times}{_B} P$ that is the identity on the base and the diffeomorphism $(p, h) \mapsto (p, p\cdot h)$ on arrows.
\end{remark}

\begin{example}
  \label{ex:4}
  Every Lie groupoid $\cl H \rra H_0$ acts on its arrow space $\cl H$ by right multiplication along the anchor map $s:\cl H \to H_0$. In other words, $P = \cl H$, $a_R = s$, and $\mu(h, h') = hh'$. Equipping $\cl H$ with this action, the surjective submersion $t:\cl H \to H_0$ becomes a principal right $\cl H$ bundle.
\end{example}

We can analogously define the left action of a Lie groupoid $\cl G$ on $P$ (denoted $\cl G \ltimes P$), and a principle left $\cl G$ bundle. Left and right actions allow us to define a bibundle.
\begin{definition}
  \label{def:17}
  A \emph{bibundle} between two groupoids $\cl G$ and $\cl H$ is a triple $(P, a_L, a_R)$, where $P$ is a manifold, $a_L:P \to G_0$ and $a_R:P \to H_0$, and
  \begin{itemize}
  \item $\cl G$ acts on $P$ from the left, with anchor $a_L$, and $\cl H$ acts on $P$ from the right, with anchor $a_R$.
  \item $a_L:P \to G_0$ is a principal right $\cl H$-bundle.
  \item $a_R$ is $\cl G$-invariant.
  \item the actions of $\cl G$ and $\cl H$ commute.
  \end{itemize}
  We also write $P:\cl G \to \cl H$.
\end{definition}

An \emph{isomorphism} of bibundles $P,Q:\cl G \to \cl H$ is a diffeomorphism $\alpha:P \to Q$ that is equivariant with respect to $\cl G-\cl H$ action, i.e.\ $\alpha(g \cdot p \cdot h) = g \cdot \alpha(p) \cdot h$, whenever this makes sense. It is also possible to define a non-associative composition of bibundles. Using Lie groupoids as the objects, bibundles as the 1-arrows, and isomorphisms of bibundles as the 2-arrows, we get a weak 2-category, sometimes denoted \textbf{Bi}. See \cite{Ler2010} for details.

\begin{definition}
  \label{def:18}
 Two Lie groupoids $\cl G$ and $\cl H$ are \emph{Morita equivalent} if there is a bibundle $P:\cl G \to \cl H$ such that $a_R:P \to H_0$ is a principal left $\cl G$-bundle. In this case we can invert $P$ to get $P^{-1}:\cl H \to \cl G$.
\end{definition}

\begin{remark}
  \label{rem:6}
  An alternative way to define Morita equivalence is through the use of refinements of Lie groupoids. We call a functor $\phi:\cl G \to \cl H$ a \emph{refinement} if
  \begin{itemize}
  \item it is \emph{essentially surjective} (ES), meaning the map $t \circ \pr_1:\cl H \tensor[_s]{\times}{_\phi} G_0 \to H_0$ is a surjective submersion, and
  \item it is \emph{fully faithful} (FF), meaning the square
    \begin{equation*}
      \begin{tikzcd}
        \cl G \ar[r, "\phi"] \ar[d, "{(s,t)}"] & \cl H \ar[d, "{(s,t)}"] \\
        G_0 \times G_0 \ar[r, "{(\phi \times \phi)}"] & H_0 \times H_0
      \end{tikzcd}
    \end{equation*}
    is a pullback.
  \end{itemize}
  Then $\cl G$ and $\cl H$ are Morita equivalent if and only if there is a groupoid $\cl K$ and refinements $\cl K \to \cl G$ and $\cl K \to \cl H$. The invertible bibundle induced by a refinement $\phi:\cl G \to \cl H$ is the pullback along $\phi_0$ of the the principal right $\cl H$-bundle $t:\cl H \to H_0$ from Example \ref{ex:4}; here the pullback of a right principal $\cl H$-bundle $\pi:P \to B$ along a map $F_0:N \to B$ is the right principal $\cl H$-bundle $\pr_1: F_0^*P \to N$, where
  \begin{itemize}
  \item $F_0^*P = P \tensor[]{\times}{_B} N$ is the pullback of the bundle $\pi:P \to B$ by $F_0:N \to B$,
  \item the anchor map for the right $\cl H$ action is $a_R \circ \pr_2$,
  \item the action is $\mu(n,p,h) = (n, p\cdot h)$.
  \end{itemize}
 \end{remark}

\subsection{The holonomy groupoid}
\label{sec:class-lie-group}

For this section we use \cite{MM2003}. We call a Lie groupoid $\cl G \rra G_0$ $\emph{source-connected}$ if the source-fibers are all connected.  As with Lie groups, we can consider the \emph{source-connected identity component} $\cl G^\circ \rra G_0$ of a Lie groupoid. The arrows in $\cl G^\circ$ are those arrows $g \in \cl G$ such that $g$ and $1_{s(g)}$ belong to the same component of $s^{-1}(s(g))$. This is an open source-connected subgroupoid of $\cl G$ (see \cite{M2005}).

A Lie groupoid is \emph{\'{e}tale} if the source map is \'{e}tale. To each \'{e}tale Lie groupoid $\cl G \rra G_0$, we have an associated pseudogroup $\Psi(\cl G)$ on $G_0$ (see Appendix \ref{sec:pseudogroups} for more on pseudogroups). This is defined in terms of bisections of $\cl G$. A \emph{bisection} of $\cl G$ is a local section $\sigma$ of $s$ such that $t \circ \sigma$ is a diffeomorphism. Then we set
\begin{equation*}
  \Psi(\cl G) := \{t\circ \sigma \mid \sigma \text{ is a bisection of } \cl G\}.
\end{equation*}

Now fix a regular folaition $\cl F$ on $M$. Its \emph{holonomy groupoid} is a Lie groupoid $\Hol\rra M$, with arrows $x \mapsto y$ the holonomy classes of leafwise paths, and multiplication given by concatenation of paths. $\Hol$ is source-connected.

Fix a complete transversal $\iota:S \hookrightarrow M$, i.e.\ an embedded submanifold of $M$ meeting every leaf transversally. We can pull back $\Hol$ along $\iota$ (equivalently, restrict it to $S$) to get $\Hol_S \rra S$, a Lie groupoid whose arrows are those arrows in $\Hol$ with endpoints in $S$. This Lie groupoid is \'{e}tale and Morita equivalent\footnote{the inclusion is a refinement, hence induces a Morita equivalence. See Remark \ref{rem:6}.} to $\Hol$. We can then consider the associated pseudogroup $\Psi(\Hol_S)$ on $S$. This is exactly the classical \emph{holonomy pseudogroup} associated to $\cl F$, for which the next result is fundamental.

\begin{theorem}
  \label{thm:4}
  For a regular foliation $\cl F$ with complete transversal $S$, the holonomy pseudogroup $\Psi(\Hol_S)$ is countably generated.
\end{theorem}

\section{Basic de Rham complexes}
\label{sec:basic-de-rham}

\begin{definition}
  \label{def:19}
  Fix a singular foliation $(M, \cl F)$ with associated singular distribution $\Delta$. A differential form $\alpha \in \cplx{}{M}$ is $\cl F$-\emph{basic} if for every section $X \in \Gamma_{\text{loc}}(\Delta)$,
  \begin{equation}\label{eq:2}
    \iota_X\alpha = 0, \text{ and } \cl L_X\alpha = 0.
  \end{equation}
  A form satisfying the first condition is \emph{horizontal}, and the second is \emph{invariant}. Denote the set of basic forms by $\cplx{b}{M, \cl F}$.
\end{definition}

\begin{remark}
  \label{rem:7}
  We could potentially define, for any smooth singular distribution $\Delta$, a $\Delta$-basic form to be one which satisfies \eqref{eq:2} for all $X \in \Gamma_{\loc}(\Delta)$. While there are no technical problems with this definition, if $\Delta$ is not integrable, then by Corollary \ref{cor:1} there is no associated singular foliation $\cl F$, and hence no related leaf-space. Since our goal is to investigate how the basic forms capture transverse structures, it therefore does not make sense for us to define basic in this generality.
\end{remark}

Because $\cl L_X\alpha = \iota_X (d\alpha) - d(\iota_X\alpha)$, any horizontal form is invariant if and only if $\iota_X d\alpha = 0$. In other words, $\alpha \in \cplx{b}{M, \cl F}$ if and only if $\iota_X\alpha = 0$ and $\iota_X d\alpha = 0$ for all $X \in \Gamma_{\loc}(\Delta)$. But the interior derivative $\iota_X$ at $x$ depends only on
  $X_x$. We therefore conclude that to prove $\alpha$ is $\cl F$-basic, it suffices to check \eqref{eq:2} against any set of vector fields spanning $\Delta$. In particular, given a Lie algebroid $(\fk g, \rho)$, a form is $\fk g$-basic (i.e.\ \eqref{eq:2} holds for all $X \in \rho(\Gamma(\fk g))$) if and only if it is basic with respect to the induced singular foliation.

The differential and exterior product of basic forms remain basic, so $\cplx{b}{M, \cl F}$ is a subcomplex of $\cplx{}{M}$. In the presence of a Lie groupoid, there is another complex of ``basic'' forms.

\begin{definition}
  \label{def:20}
  Fix a Lie groupoid $\cl G \rra M$. A differential form $\alpha \in \cplx{}{M}$ is $\cl G$-basic if $s^*\alpha = t^*\alpha$. Denote the de Rham complex of these forms by $\cplx{b}{M, \cl G}$.
\end{definition}

\begin{remark}
  \label{rem:8}
  For a Lie group $G$ acting on $M$, there is also a notion of a $G$-basic form. This is a form $\alpha$ such that $g^*\alpha = \alpha$ for all $g \in G$, and $\iota_X\alpha$ for all $X$ tangent to the orbits of $G$. However, this is not a fundamentally different notion than that from Lie groupoids: a form is $G$-basic if and only if it is $G \ltimes M$-basic (Lemma 3.3 in \cite{W2015}).
\end{remark}

We relate the complexes of $\cl G$-basic and $\cl F_{\cl G}$-basic forms in Proposition \ref{prop:5}.
\subsection{Basic forms and regular foliations}
\label{sec:basic-forms-regular}

In this section we prove statement \ref{item:1} in the Introduction, namely that for a regular foliation $\cl F$, the pullback by the quotient $\pi:M \to M/\cl F$ is an isomorphism $\pi^*:\cplx{}{M/\cl F} \to \cplx{b}{M, \cl F}$. Another proof can be found in \cite{H2011}, but we arrived at this result independently.

To contextualize the following propositions, we will outline the proof here. The pullback is one-to-one for purely diffeological reasons (Lemma \ref{lem:1}). First, we show $\pi^*$ has image contained in $\cl F$-basic forms, so that the question is well-posed (Proposition \ref{prop:4}). Then, we show that the $\cl F$-basic forms are exactly the $\Hol$-basic forms, because $\Hol$ is a source-connected groupoid whose orbits are the leaves of $\cl F$ (Proposition \ref{prop:5}). This reduces the question to whether $\pi^*:\cplx{}{M/\cl F} \to \cplx{b}{M, \Hol}$ is an isomorphism. Now take a complete transversal $S$ to $\cl F$, to obtain the \'{e}tale holonomy groupoid $\Hol_S \rra S$, Morita equivalent to $\Hol \rra M$. Because of this equivalence, the pullback by $\pi$ is an isomorphism if and only if the pullback by $\pi_S:S \to S/\Hol_S$ is an isomorphism (Proposition \ref{prop:6}, Corollary \ref{cor:2}). But $\Hol_S$ has finitely generated pseudogroup, and we can show this implies $\pi_S^*$ is an isomorphism (Lemma \ref{lem:2}).

\begin{proposition}\
  \label{prop:4}
  \begin{enumeratea}
  \item For a singular foliation, every pullback of a form on $M/\cl F$ is $\cl F$-basic.
    \item For a Lie groupoid, every pullback of a form on $M/\cl G$ is $\cl G$-basic.
  \end{enumeratea}
\end{proposition}

\begin{proof}\
  \begin{enumeratea}
  \item Suppose $\beta \in \Omega^k(M/\cl F)$ and set $\alpha := \pi^*\beta$. For any $X \in \Gamma_{\text{loc}}(\Delta)$, we must show $\iota_X \alpha = 0$ and $\cl{L}_X\alpha = 0$.  By Proposition \ref{prop:3}, for any plots $P,Q:U \to M$ such that $\pi \circ P = \pi \circ Q$, we have $P^*\alpha = Q^*\alpha$.  We can in fact replace the domain of $P,Q$ with any manifold, and thus set
     \begin{align*}
       P &:= \Phi:\cl D \to M \text{ to be the flow of } X, \quad (t, p) \mapsto \Phi^t(p) = \Phi^{(p)}(t)\\
       Q &:= \pr_2:\cl D \to M, \quad (t, p) \mapsto p.
     \end{align*}
     Since $\Phi^t(p)$ and $p$ always share a leaf, $\pi \circ \Phi = \pi \circ \pr_2$, and therefore $\Phi^*\alpha = \pr_2^*\alpha$. For $(t, p) \in \cl D$, we identify $T_{(t, p)}\cl D$ with $\RR \oplus T_pM$. Under this identification, for every $k$ vectors $v_1, \ldots, v_k\in T_pM$, at $t = 0$ we have

     \begin{align*}
       (\Phi^*\alpha)_{(0, p)}(1 \oplus v_1, v_2, \ldots, v_k) &= 
       \alpha_p(X_p +v_1, v_2, \ldots, v_k)\\
       (\pr_2^*\alpha)_{(0, p)}(1 \oplus v_1, v_2, \ldots, v_k) &= \alpha_p(v_1, v_2, \ldots, v_k).
     \end{align*}

     This implies $\iota_{X_p}\alpha_p = 0$.  We also have for each $t$, at $\vec v = (v_1, \ldots, v_k)$,

     \begin{equation*}
       \alpha_p(\vec v) = (\pr_2^*\alpha)_{(p, t)}(\vec v) = (\Phi^*\alpha)_{(t, p)}(\vec v) = ((\Phi^t)^*\alpha)_p(\vec v),
     \end{equation*}

     hence 
     \begin{equation*}
       0 = \frac{d}{dt}\Big|_{t=0} \alpha_p(\vec v) = \frac{d}{dt}\Big|_{t=0} ((\Phi^t)^*\alpha)_p(\vec v) = (\cl L_X\alpha)_p(\vec v)       
     \end{equation*}

     Therefore $\iota_X\alpha$ and $\cl L_X\alpha$ vanish.
    \item This argument is Corollary 3.6 in \cite{W2015}. Suppose $\beta \in \Omega^k(M/\cl G)$ and set $\alpha := \pi^*\beta$. Then by Remark \ref{rem:2}, $\alpha$ is basic with respect to the relation groupoid $M \tensor[]{\times}{_\pi} M \rra M$. So then $\pr_1^*\alpha - \pr_2^*\alpha = 0$. Pulling this back by $(s,t):\cl G \to M \tensor[]{\times}{_\pi} M$ yields $0 = (s,t)^*(\pr_1^*\alpha - \pr_2^*\alpha)$, and thus $s^*\alpha = t^*\alpha$.  
  \end{enumeratea}
\end{proof}

For this next proposition, we adapt the proof from \cite{HSZ2019}, which deals only with regular foliations.

\begin{proposition}
  \label{prop:5}
  Let $\cl G \rra M$ be a Lie groupoid with associated singular foliation $\cl F$. Every $\cl G$-basic form is $\cl F$-basic, and every $\cl F$-basic form is $\cl G^\circ$-basic.
\end{proposition}

\begin{proof}
  Let $(\fk g, \rho)$ be the Lie algebroid of $\cl G$, where $\fk g = (\ker ds)|_M$ and $\rho = dt$. To to prove a form is $\cl F$-basic, it suffices to only test against vector fields in $\rho(\Gamma(\fk g))$. For a section $\sigma$ of $\fk g$, denote the corresponding right-invariant vector field on $\cl G$ by $\tilde \sigma$. For an arrow $g:x \mapsto y$, we have

  \begin{align*}
    ds(\tilde \sigma_g) &= ds(dR_g(\tilde \sigma_{1_y})) &&\text{by definition of }\tilde \sigma \\
                                                        &= d(s \circ R_g)(\tilde \sigma_{1_y}) \\
    &= 0 &&\text{because } s \circ R_g \text{ is constant}.
  \end{align*}
  Similarly, $dt(\tilde \sigma_g) = dt(\tilde \sigma_{1_y})$.  Therefore
  \begin{equation*}
    \tilde \sigma \sim_s 0 \text{ and } \tilde \sigma \sim_t \rho(\sigma),
  \end{equation*}
  which implies that for any form $\alpha$ on $M$,
  \begin{equation}
    \label{eq:3}
    \cl{L}_{\tilde{\sigma}}s^*\alpha = 0, \text{ and } \cl{L}_{\tilde{\sigma}} t^*\alpha = t^*\cl{L}_{\rho(\sigma)}\alpha.
  \end{equation}
  Now we take the splitting $T\cl G|_M = \fk g \oplus TM$, which allows us to write
  \begin{equation}
    \label{eq:4}
    \begin{split}
      (s^*\alpha)_{1_x}(\xi_1+w_1, \ldots, \xi_k+w_k) &= \alpha_x(w_1, \ldots, w_k) \\
      (t^*\alpha)_{1_x}(\xi_1+w_1, \ldots, \xi_k+w_k) &= \alpha_x(\rho(\xi_1)+w_1, \ldots, \rho(\xi_k)+w_x).
    \end{split}
  \end{equation}
  We will now prove the proposition.
  \begin{itemize}
  \item Suppose $\alpha$ is $\cl G$-basic. Then $s^*\alpha = t^*\alpha$, and so setting $v_i = 0$ for $i > 1$ and $w_1 = 0$ in equation \eqref{eq:4}, we get
    \begin{equation*}
      \alpha_x(0, w_2, \ldots, w_k) = (\iota_{\rho(\xi_1)}\alpha)_x(w_2, \ldots, w_k).
    \end{equation*}
    The left side is always $0$, so we get $\iota_{\rho(\sigma)}\alpha = 0$. For invariance, by equation \eqref{eq:3},
    \begin{equation*}
      0 = \cl L_{\tilde \sigma}s^*\alpha = \cl L_{\tilde \sigma} t^*\alpha = t^*\cl L_{\rho(\sigma)}\alpha,
    \end{equation*}
    and since $t$ is a submersion, we get $\cl L_{\rho(\sigma)}\alpha = 0$. Therefore $\alpha$ is $\cl F$-basic.
  \item Now suppose $\alpha$ is $\cl F$-basic. The fact $\iota_{\rho(\sigma)}\alpha = 0$ implies
    \begin{equation*}
      \alpha_x(\rho(\xi_1)+w_1, \ldots, \rho(\xi_k)+w_k) = \alpha_x(w_1, \ldots, w_k),
    \end{equation*}
    so by equation \eqref{eq:4} we get that $s^*\alpha = t^*\alpha$ at points $1_x \in M$. Furthermore the assumption $\cl L_{\rho(\sigma)}\alpha = 0$ combined with equation \eqref{eq:3} gives
    \begin{equation*}
      \cl L_{\tilde \sigma} s^*a = 0 = t^*\cl L_{\rho(\sigma)}\alpha = \cl L_{\tilde \sigma}t^*\alpha.
    \end{equation*}
    Therefore $s^*\alpha$ and $t^*\alpha$ are invariant under the flows of all the $\tilde \sigma$. Now notice that these vector fields span $\ker ds$, which is an involutive subbundle of $T\cl G$ that foliates $\cl G$ by the connected components of the source-fibers. In particular, we can connect any arrow in the component $1_x$ to $1_x$ by travelling along the flows of the $\tilde \sigma$.  Therefore $s^*\alpha = t^*\alpha$ on the union of connected components of the identity arrows. This is exactly the arrow space of $\cl G^\circ \rra M$, thus $\alpha$ is $\cl G^\circ$-basic. 
  \end{itemize}
\end{proof}

The following proposition and its proof are from \cite{W2015}, Proposition 3.9, but can also be found as Lemma 5.3.8 in \cite{HSZ2019}.
\begin{proposition}
  \label{prop:6}
  Let $\cl{G}$ and $\cl{H}$ be Morita equivalent Lie groupoids, witnessed by an invertible bibundle $P:\cl{G} \to \cl{H}$. There is an isomorphism $P^*:\cplx{b}{H_0, \cl H} \to \cplx{b}{G_0, \cl G}$ defined uniquely by the condition that $a_R^*\alpha = a_L^*P^*\alpha$ (where $a_R$ and $a_L$ are the anchor maps for the actions).
\end{proposition}

\begin{proof}
      The bibundle $P:\cl G \to \cl H$ gives the following commutative diagram.

    \begin{equation*}
      \begin{tikzcd}
        & \cl G \tensor[_s]{\times}{_{a_L}} P \ar[dr, shift left, "\mu_L"] \ar[dr, shift right, "\pr_2"'] \ar[dl, "\pr_1"] & & P\tensor[_{a_R}]{\times}{_t} \cl H \ar[dr, "\pr_2"] \ar[dl, shift left, outer sep=3, "\mu_R"'] \ar[dl, shift right, outer sep=3, "\pr_1" ] & \\
        \cl G \ar[dr, shift left, "t"] \ar[dr, shift right, "s"'] & &P \ar[dr, "a_R"] \ar[dl, "a_L"] & & \cl H \ar[dl, shift left, outer sep=3, "s"'] \ar[dl, shift right, outer sep=3, "t"] \\
        & G_0 & & H_0
      \end{tikzcd}
    \end{equation*}

    Let $\alpha \in \cplx{b}{M, \cl H}$. The pullback $a_R^*\alpha$ is $P \rtimes H$-basic, since by commutativity and $\cl H$-basic,
    \begin{equation*}
      \mu_R^* a_R^*\alpha = \pr_2^*s^*\alpha = \pr_2^*t^*\alpha = \pr_1^*a_R^*\alpha.
    \end{equation*}
    By Remark \ref{rem:5}, the Lie groupoid $P \rtimes \cl H$ is isomorphic to $P \tensor[]{\times}{_{G_0}} P$, so then $a_R^*\alpha$ is also $P \tensor[]{\times}{_{G_0}} P$-basic. By Remark \ref{rem:2}, this is equivalent to $a_R^*\alpha = a_L^*\beta$, for some $\beta \in \cplx{}{G_0}$. Note $\beta$ is unique because $a_L$ is a surjective submersion. We say $P^*\alpha := \beta$. We claim $\beta$ is $\cl G$-basic. First, observe that by commutativity, $a_R^*\alpha = a_L^*\beta$, and $a_R^*\alpha$ being $\cl G \ltimes P$-basic:
    \begin{equation*}
      \pr_1^* s^*\beta = \pr_2^*a_L^*\beta = \pr_2^* a_R^*\alpha = \mu_L^* a_R^*\alpha = \mu_L^* a_L^*\beta = \pr_1^*t^*\beta.
    \end{equation*}
    Then $s^*\beta = t^*\beta$ because $a_L$, hence $\pr_1$, is a surjective submersion.

    It is evident that $P^*$ is well-defined, and a homomorphism of complexes. To see it is an isomorphism, observe that its inverse is $(P^{-1})^*$.
  \end{proof}

  \begin{corollary}
    \label{cor:2}
    The pullback $\pi_{\cl G}^*$ surjects onto $\cl G$-basic forms if and only if $\pi_{\cl H}^*$ is onto $\cl H$-basic forms.
  \end{corollary}

  \begin{proof}
    The map $\Psi:H_0/\cl H \to G_0/\cl G$ defined by $\pi_{\cl H}(y) \mapsto \pi_{\cl G}(a_L(a_R^{-1}(y)))$ is a well-defined diffeological diffeomorphism (\cite{W2015}, Theorem 3.8). Now assume $\pi_{\cl G}^*$ surjects, and take $\alpha \in \cplx{b}{M, \cl H}$. Set $\beta := P^*\alpha$. By assumption there is some $\ol \beta \in \cplx{}{G_0/\cl G}$ with $\pi_{\cl G}^*\ol \beta = \beta$. Then
    \begin{equation*}
      a_L^*\beta = a_L^*\pi_{\cl G}^*\ol \beta = a_R^*(\pi_{\cl H}^*\Psi^*\ol \beta).
    \end{equation*}
    But $a_L^* \beta$ is also $a_R^*\alpha$, and since $a_R$ is a surjective submersion, we get $\alpha = \pi_{\cl H}^*\Psi^*\ol \beta$. In other words, $\pi_{\cl H}^*$ also surjects. For the converse direction, work with $(P^{-1})^*$.
  \end{proof}

  Now, we have a final lemma.

  \begin{lemma}
    \label{lem:2}
     Let $\cl G \rra M$ be an \'{e}tale Lie groupoid with countably generated associated pseudogroup $\Psi(\cl G)$. Then the pullback by $\pi:M \to M/\cl G$ is onto $\cl G$-basic forms.
   \end{lemma}
   
   \begin{proof}
         See Appendix \ref{sec:pseudogroups} for the relevant facts about pseudogroups. Let $\alpha \in \cplx{b}{M, \cl G}$. Take $P, Q:U \to M$ such that $\pi \circ P = \pi \circ Q$. By Proposition \ref{prop:3}, it suffices to show $P^*\alpha = Q^*\alpha$.  First, note that $\alpha$ is $\Psi(\cl G)$-invariant, since for any $f = t \circ \sigma \in \Psi(\cl G)$, we have
    \begin{equation*}
      f^*\alpha =  \sigma^*t^*\alpha = \sigma^*s^*\alpha = \id^*\alpha = \alpha.
    \end{equation*}
    Say $\{f_i\}_{i=1}^\infty$ generates $\Psi(\cl G)$. For each tuple $I:=(i_1, \ldots, i_N)$, define $f_I := f_{i_1} \circ \cdots \circ f_{i_N}$, and set $C_I:= \{r \in U \mid f_I(P(r)) = Q(r)\}$.  Each $C_I$ is closed in $U$, and we claim $U \subseteq \bigcup_I C_I$ (hence equality holds). Indeed, for any $r \in U$, we have $\pi(P(r)) = \pi(Q(r))$, so there is an arrow $P(r) \mapsto Q(r)$. Taking its image under $\Eff$ (see the discussion preceding Definition \ref{def:24}), this gives some $f \in \Psi(\cl G)$ such that $f(P(r)) = Q(r)$. Using our generating family for $\Psi(\cl G)$, we can write $f = f_I$ locally near $r$ for some $I$, hence $r \in C_I$.

    By the Baire category theorem, $\bigcup_I \text{int}(C_I)$ is open and dense in $U$. But on each $\text{int}(C_I)$, we have $f_I \circ P = Q$, so by $\Psi(\cl G)$-invariance of $\alpha$,
    \begin{equation*}
      P^*\alpha = P^*f^*\alpha = (f_I \circ P)^*\alpha = Q^*\alpha.
    \end{equation*}
    As this holds on the open dense subset $\bigcup_I \text{int}(C_I)$, by continuity $P^*\alpha = Q^*\alpha$ on all of $U$, as required.
  \end{proof}

We may now give the formal statement and proof of our result \ref{item:1}.

  \begin{theorem}
    \label{thm:5}
    Suppose $(M, \cl F)$ is a regular foliation. Equip $M$ and $M/\cl F$ with their manifold and quotient diffeology, respectively. The quotient map $\pi:M \to M/\cl F$ is diffeologically smooth, and its pullback restricts to an isomorphism from diffeological forms on $M/\cl F$ to $\cl F$-basic forms on $M$. In other words, $\pi^*:\cplx{}{M/\cl{F}} \to \cplx{b}{M, \cl F}$ is an isomorphism (c.f.\ \cite{H2011}).
  \end{theorem}

  \begin{proof}
     The pullback $\pi^*$ is injective by Lemma \ref{lem:1}, and maps into basic forms by Proposition \ref{prop:4} (a). It remains is to show $\pi^*$ is surjective.

The foliation $\cl F$ is induced by its holonomy groupoid $\Hol \rra M$. As $\Hol$ is source-connected, by Proposition \ref{prop:5} the $\cl F$-basic and $\Hol$-basic forms coincide. Since $M/\Hol = M/\cl F$, it therefore suffices to show that $\pi^*$ is onto $\Hol$-basic forms.

  Fix a complete transversal $S$ to $\cl F$. The restriction $\Hol_S\rra S$ of $\Hol$ to $S$ is Morita equivalent to $\Hol$ (see Section \ref{sec:class-lie-group}). Therefore by Corollary \ref{cor:2}, to show $\pi^*$ surjects onto $\Hol$-basic forms, we may instead show the pullback by $\pi_S:S \to S/\Hol_S$ is onto $\Hol_S$-basic forms. The groupoid $\Hol_S \rra S$ is \'{e}tale, and its associated pseudogroup is countably generated by Theorem \ref{thm:4}. We apply Lemma \ref{lem:2} to complete the proof.
\end{proof}

\begin{corollary}
  \label{cor:3}
  If $\cl G \rra M$ is a Lie groupoid Morita equivalent to an \'{e}tale groupoid,\footnote{such groupoids are often called \emph{foliation groupoids}.} the pullback by $\pi:M \to M/\cl G$ is an isomorphism onto $\cl G$-basic forms.
\end{corollary}
\begin{proof}
  Because $\cl G$ is Morita equivalent to an \'{e}tale groupoid, its associated isotropy groups are discrete (Proposition 5.20 in \cite{MM2003}). Therefore the singular foliation $\cl F$ associated to $\cl G$ is regular, and Theorem \ref{thm:5} every $\cl F$-basic form is a pullback from the quotient. But the $\cl F$-basic forms are exactly the $\cl G$-basic forms by source-connectedness and Proposition \ref{prop:5}. So the pullback must be onto $\cl G$-basic forms as well.
\end{proof}

\subsection{Singular foliations decomposed by dimension}
\label{sec:sing-foli-strat}

We now prove statement \ref{item:2} from the Introduction. We begin by establishing some terminology.

\begin{definition}
  \label{def:21}
  A singular foliation $(M, \cl F)$ is \emph{decomposed by dimension} if the sets $M_{=k}$, consisting of points in leaves of dimension $k$, are diffeological submanifolds of $M$, perhaps with components of varying dimension.
\end{definition}

Take an arbitrary singular foliation $(M, \cl F)$ with associated distribution $\Delta$. For $* \in \{=,\geq,>,<,\leq,\neq\}$, set
\begin{equation*}
  \cl F_{*k} := \{L \in \cl F \mid \dim L * k\}.
\end{equation*}

\begin{lemma}
  \label{lem:3}
  If $M_{*k}$ is a diffeological submanifold of $M$, then $(M_{*k}, \cl F_{*k})$ is a singular foliation.
\end{lemma}

\begin{proof}

By the Stefan-Sussmann Theorem \ref{thm:3}, we may take a collection of arrows $\cl A$ such that the orbits of $\Psi \cl A$ are the leaves of $\cl F$. Let $\cl A'$ consist of the arrows in $\cl A$ restricted to $M_{*k}$. This is a collection of arrows, and so by Stefan's Theorem \ref{thm:2}, the orbits of $\Psi \cl A'$ form a singular foliation of $M_{*k}$. But these orbits are exactly the elements of $\cl F_{*k}$, hence $\cl F_{*k}$ is indeed a singular foliation.
\end{proof}

\begin{lemma}
  \label{lem:4}
  Suppose $M_{*k}$ is a diffeological submanifold of $M$, and $\alpha \in \cplx{b}{M, \cl F}$. Then $\alpha' := \alpha|_{M_{*k}}$ is $\cl F_{*k}$-basic.
\end{lemma}

\begin{proof}
  Take $\cl A$ and $\cl A'$ as in the previous lemma. Then $\alpha$ is $\cl F$-invariant, hence $\Psi \cl A$-invariant. This means $\alpha'$ is $\Psi \cl A'$-invariant, hence $\cl F_{*k}$-invariant. As for horizontal, we must take a slightly pedantic approach. Denote the inclusions
  \begin{equation*}
  \begin{tikzcd}
    L \ar[r, hook, "\iota'"] \ar[rr, hook, bend left, "\iota''"] & M_{*k} \ar[r, hook, "\iota"] & M.
  \end{tikzcd}  
\end{equation*}

Both $\iota'$ and $\iota''$ are smooth immersions, since $L$ is weakly-embedded, but $\iota$ is merely smooth. Suppose $v \in T_xL$. We want to show $\iota_{\iota'_*v}\alpha' = 0$. Compute
\begin{equation*}
  \iota_{\iota'_*v}\alpha' = \alpha'(\iota_*'v, \cdot) = \alpha(\iota_* \iota'_* v, \iota_*\cdot) = \alpha(\iota''_*v, \iota_*\cdot).
\end{equation*}
But the right side is 0, because $\alpha$ is $\cl F$-horizontal.
\end{proof}

\begin{theorem}
  \label{thm:6}
  Suppose the singular foliation $(M, \cl F)$ is decomposed by dimension. Equip $M$ and $M/\cl F$ with the manifold and quotient diffeology, respectively. The quotient map $\pi:M \to M/\cl F$ is diffeologically smooth, and pulling back by the quotient is an isomorphism from diffeological forms on $M/\cl F$ to $\cl F$-basic forms on $M$. In other words, $\pi^*:\cplx{}{M/\cl F} \to \cplx{b}{M, \cl F}$ is an isomorphism.
\end{theorem}

\begin{proof}
  By Proposition \ref{prop:4}, $\pi^*$ maps into $\cl F$-basic forms, and by Lemma \ref{lem:1} $\pi^*$ is injective. It remains to show $\pi^*$ is surjective. Let $k_{\text{max}}$ denote the highest dimension of the leaves of $\cl F$.  Equip $M_{*k}$ with the singular foliation $\cl F_{*k}$. Fix $\alpha \in \cplx{b}{M, \cl F}$. Consider the statement
  \begin{equation}\tag*{$S(k)$}\label{eq:5}
      \alpha|_{M_{\geq k}} \text{ is the pullback of a form on the quotient } M_{\geq k}/\cl F_{\geq k}.
      \end{equation}
      If $S(k_{\text{max}})$ holds, and if $S(k+1) \implies S(k)$, then $S(0)$ holds, which is what we want to prove. Now, $S(k_{\text{max}})$ is equivalent to: there is a form $\beta$ on $M_{= k_{\text{max}}}/\cl F_{=k_{\text{max}}}$ such that $\pi^*\beta = \alpha|_{M_{=k_{\text{max}}}}$. But $(M_{=k_{\text{max}}}, \cl F_{=k_{\text{max}}})$ is a regular foliation, and $\alpha|_{M_{=k_{\text{max}}}}$ is $\cl F_{=k_{\text{max}}}$-basic by Lemma \ref{lem:4}, so $S(k_{\text{max}})$ holds by Theorem \ref{thm:5}.
      
Now assume $S(k+1)$. We will use Proposition \ref{prop:3} to conclude $S(k)$. Let $P,Q:U \to M_{\geq k}$ be plots such that $\pi \circ P = \pi \circ Q$. Set
\begin{align*}
  A &:= P^{-1}(M_{\geq k+1}) \quad (=Q^{-1}(M_{\geq k+1})), \text{ which is open in }U \\
  B &:= P^{-1}(M_{=k}) \quad (=Q^{-1}(M_{=k})).
\end{align*}
Then $U = A \sqcup B$, and so $U = \ol A \cup \text{int}(B)$. We will show $\alpha(P) = \alpha(Q)$ first on $A$, and then on $\text{int}(B)$. By continuity, this yields $\alpha(P) = \alpha(Q)$ on $\ol A \cup \text{int}(B) = U$.
\begin{itemize}
\item \textbf{For $A$}: The plots $P$ and $Q$ restrict to maps $P',Q':A \to M_{\geq k+1}$, which are smooth maps between the $U$-open set $A$ and the $M$-open set $M_{\geq k+1}$. We have $\pi \circ P' = \pi \circ Q'$, and we are assuming $S(k+1)$. Therefore by Proposition \ref{prop:3},
  \begin{equation*}
    \alpha|_{M_{\geq k+1}}(P') = \alpha|_{M_{\geq k+1}}(Q'), \text{ which implies } \alpha(P)|_A = \alpha(Q)|_A.
  \end{equation*}
\item \smallskip\noindent\textbf{For }$\text{int}(B)$: We may assume $\text{int}(B)$ is non-empty. Fix $r \in \text{int}(B)$. By the Stefan-Sussmann Theorem (Remark \ref{rem:1}), we can take $M$-open sets $V$ about $P(r)$ and $V'$ about $Q(r)$, and a diffeomorphism $\xi:V \to V'$ respecting the leaves, such that $\xi(P(r)) = Q(r)$. Let $U'$ be an open connected neighbourhood of $r$ such that $U' \subseteq \text{int}(B) \cap P^{-1}(V)$. Define
  \begin{align*}
    P':U' &\to M_{= k}, \quad r' \mapsto \xi(P(r')) \\
    Q':U' &\to M_{= k}, \quad r' \mapsto Q(r').
  \end{align*}
These are continuous. As $U'$ is connected, its images under $P'$ and $Q'$ are connected, and in particular lie in a connected component of $M_{=k}$. Because $P'(r) = Q'(r)$, in fact both $P'$ and $Q'$ map into a single connected component $M_{=k}^\circ$ of $M_{=k}$. Denote the restrictions of $P'$ and $Q'$ to maps $U' \to M_{=k}^\circ$ again by $P'$ and $Q'$. We may view $P'$ and $Q'$ as smooth maps from the $U$-open set $U'$ to the manifold $M_{=k}^\circ$. Observe that
  \begin{equation*}
    \pi(P'(r')) = \pi(\xi(P(r'))) = \pi(P(r')) = \pi(Q'(r')),
  \end{equation*}
  so $\pi \circ P' = \pi \circ Q'$. Also, $\alpha|_{M_{=k}^\circ}$ is basic with respect to $\cl F_{=k}$ by Lemma \ref{lem:4}. Therefore by Theorem \ref{thm:5} $\alpha|_{M_{=k}^\circ}$ is the pullback of some form on $M_{=k}^\circ / \cl F_{=k}$. By Proposition \ref{prop:3}, we get
  \begin{equation*}
    \alpha|_{M_{=k}^\circ}(P') = \alpha|_{M_{=k}^\circ}(Q'), \text{ which implies } \alpha(\xi \circ P)|_{U'} = \alpha(Q)|_{U'}.
  \end{equation*}

  Because $\alpha$ is $\cl F$-basic, $\alpha(\xi \circ P) = \alpha(P)$. As $r$ was arbitrary, we can conclude that $\alpha(P)|_{\text{int}(B)} = \alpha(Q)|_{\text{int}(B)}$.
\end{itemize}

Therefore, for any two plots $P,Q$ of $M$ such that $\pi \circ P = \pi \circ Q$, we have proved $\alpha(P) = \alpha(Q)$. By Proposition \ref{prop:3}, $\alpha$ is the pullback of some diffeological form on $M/\cl F$.
\end{proof}

\begin{corollary}
  \label{cor:4}
  If $\cl G \rra M$ is a source-connected linearizable Lie groupoid (see Appendix \ref{sec:line-strat-dimens}), then $\pi^*:\cplx{}{M/\cl G} \to \cplx{b}{M, \cl G}$ is an isomorphism.
\end{corollary}

\begin{proof}
  Let $\cl F_{\cl G}$ be the singular folaition consisting of the orbits of $\cl G$. By Proposition \ref{prop:5}, $\cl G$ linearizable implies $\cl F_{\cl G}$ is decomposed by dimension. Hence by Theorem \ref{thm:6}, $\pi$ is onto $\cl F_{\cl G}$-basic forms. By source-connectedness, and Proposition \ref{prop:5}, $\cplx{b}{M, \cl F_{\cl G}} = \cplx{b}{M, \cl G}$, which completes the proof.
\end{proof}

Here we discuss how this result fits into existing literature. Crainic and Struchiner \cite{CS2013} proved every proper Lie groupoid is linearizable. In this case, the fact $\pi^*$ is an isomorphism was also proved by Watts \cite{W2015}. Watts relied on properness to ensure compactness of the isotropy groups $G_x$, and then applied his and Karshon's \cite{KW2016} earlier result that $\pi^*$ is an isomorphism whenever $\cl G$ is the action groupoid of a Lie group action on $M$ with properly acting identity component. This earlier result also relied on compactness of the $G_x$ for a proper Lie group action.

To finish, we show that $\pi^*$ is an isomorphism for a much broader class of singular foliations than those decomposed by dimension.

\begin{theorem}
  \label{thm:7}
  Suppose the singular foliation $(M, \cl F)$ is such that the pullback by $\pi:M_{>0} \to M_{>0}/\cl F_{>0}$ is an isomorphism $\pi^*:\cplx{}{M_{>0}} \to \cplx{b}{M_{>0}, \cl F_{>0}}$. Then the pullback by the quotient $\pi:M \to M/\cl F$ is an isomorphism $\pi^*:\cplx{}{M} \to \cplx{b}{M, \cl F}$.
\end{theorem}

\begin{proof}
  This proof uses the same ideas as in the previous theorem. Letting $\alpha \in \cplx{b}{M, \cl F}$, all we need to show is that $\alpha$ comes from the quotient. We use Proposition \ref{prop:3}. Let $P, Q:U \to M$ be plots such that $\pi \circ P = \pi \circ Q$. Set
  \begin{align*}
  A &:= P^{-1}(M_{>0}) \quad (=Q^{-1}(M_{>0})), \text{ which is open in }U \\
  B &:= P^{-1}(M_{=0}) \quad (=Q^{-1}(M_{=0})).
  \end{align*}
  As before, to show $\alpha(P) = \alpha(Q)$, it suffices to show equality on $A$ and on $\text{int}(B)$. For $A$, this is a direct result of the assumption and the fact $M_{>0}$ is open, so that $(M_{>0}, \cl F_{>0})$ is a singular foliation of a manifold. For $r \in \text{int}(B)$, note that $\pi \circ P(r) = \pi \circ Q(r)$ reduces to $P(r) = Q(r)$, because the $0$-leaves are just points. Therefore $P = Q$ on $\text{int}(B)$, and $\alpha(P) = \alpha(Q)$ on $\text{int}(B)$ follows immediately.
\end{proof}

Some consequences of the previous theorem are that $\pi^*$ is an isomorphism when
\begin{itemize}
\item $\Delta_{\cl F}$ is spanned by a single vector field. For instance, every singular foliation of $\RR$.
\item the singular leaves of $(M, \cl F)$ are all points (this implies $M_{=k}$ is open for $k \geq 1$). Examples of this type include the coadjoint action of $\text{SL}(2, \RR)$, and more generally any singular foliation induced by any Poisson structure on a 2 or 3-dimensional manifold.
\end{itemize}

Both examples can come from Lie groupoids which are not linearizable, and perhaps some are not induced by any Lie algebroid at all.
\appendix
\section{Linearization of Lie groupoids and decomposition by dimension}
\label{sec:line-strat-dimens}

Here we review the linearization of a Lie groupoid about an orbit. Our sources are \cite{CS2013} and \cite{F2015}. Refer to Section \ref{sec:lie-algebr-group} for the relevant facts about Lie groupoids. Fix a Lie groupoid $\cl G \rra M$, and orbit $\cl O$ through $x \in M$. We form the restricted groupoid $\cl G_{\cl O} \rra \cl O$, whose arrows are those arrows in $\cl G$ which begin (and end) in $\cl O$, and all the structure maps are induced from $\cl G \rra M$. This is a Lie groupoid.

Let $\nu(\cl O)$ denote the normal bundle over $\cl O$, and similarly for $\nu(\cl G_{\cl O})$.   We can form the groupoid $\nu(\cl G_{\cl O}) \rra \nu(\cl O)$ using the short exact sequence of groupoids:
  \begin{equation*}
    \begin{tikzcd}
      1 \ar[r] & T\cl G_{\cl O} \ar[d, shift left] \ar[d, shift right] \ar[r] & T\cl G \ar[d, shift left] \ar[d, shift right] \ar[r] & \nu(\cl G_{\cl O}) \ar[d, shift left] \ar[d, shift right] \ar[r] & 1 \\
      0 \ar[r] & T\cl O \ar[r] & TM \ar[r] & \nu(\cl O) \ar[r] & 0
    \end{tikzcd}
  \end{equation*}

  \begin{definition}
  \label{def:22}
  A Lie groupoid is \emph{linearizable} at an orbit $\cl O$ (through $x$) if there is an open neighbourhood $U$ of $\cl O \subseteq M$ and an open neighbourhood $V$ of $\cl O \subseteq \nu(\cl O)$ (viewing $\cl O$ as the image of the zero-section), and an isomorphism of the Lie groupoids $\cl G|_U\rra U$ and $\nu(\cl G_{\cl O}) |_V \rra V$, which is the identity on $\cl G_{\cl O} \rra \cl O$.
\end{definition}

We next present two alternative descriptions of the linear model. First, $\cl G_{\cl O}$ acts on $\nu(\cl O)$ from the left with anchor $\pi$ by. 

    \begin{equation*}
    \mu_L(g, [v]) = g \cdot [v] := [dt_g(\tilde v)] \text{ where } \tilde v \in T_g\cl G \text{ satisfies } ds_g(\tilde v) = v.  
  \end{equation*}

    The action is well-defined, and yields the action groupoid $\cl G_{\cl O} \ltimes \nu(\cl O)$. One can check this is isomorphic to $\nu(\cl G_{\cl O})$ by $[\tilde v_g] \mapsto (g, [ds_g \tilde v])$. Our second description is point-wise. Consider the two principal $G_x$-bundles:
    \begin{align*}
      t:P_x &\to \cl O \\
      P_x \times P_x &\to \cl G_{\cl O}, \quad (g, h) \mapsto gh^{-1}.
    \end{align*}

    The action $\mu_L$ provides a left action of $G_x$ on $\nu_x\cl O$. Then, we can form the two associated bundles,
    \begin{align*}
      P_x \times_{G_x} \nu_x \cl O, \quad \text{and} \quad (P_x \times P_x) \times_{G_x} \nu_x \cl O. 
    \end{align*}
      Both bundles above are isomorphic to $\nu(\cl O)$ and $\cl G_{\cl O} \ltimes \nu(\cl O)$, respectively: the isomorphisms are $[k, w] \mapsto k \cdot w$, and $[k, k', w] \mapsto (k(k')^{-1}, k \cdot w)$. Using these identifications to transport the groupoid structure of $\nu(\cl G_{\cl O})$, we obtain the Lie groupoid $(P_x \times P_x) \times_{G_x} \nu_x \cl O \rra P_x \times_{G_x} \nu_x \cl O$, which is a pointwise version of the linear model.

    \begin{example}
      \label{ex:5}
      Consider a Lie group $G$ acting on a manifold $M$, with orbit $\cl O$ through $x$.  Take $\cl G$ to be the action groupoid $G\ltimes M$. Its pointwise linear model is given by the left $G$-action on $G \tensor[]{\times}{_{G_x}} \nu_x\cl O$.  If the action is locally proper at $x$ (see Definition 2.3.2 in \cite{DK2000}), the Tube Theorem (for example, Theorem 2.4.1 in \cite{DK2000}) gives a $G$-equivariant diffeomorphism from $G \tensor[]{\times}{_{G_x}} D$, where $D$ is a $G_x$-invariant open neighbourhood of $0$ in $\nu_x\cl O$, to some saturated neighbourhood $U$ of $\cl O$ in $M$.  Therefore if the action is locally proper at $x$, it is linearizable in the sense of Definition \ref{def:22} at $\cl O$, and furthermore we can take $U$ and $V$ saturated. Palais \cite{P1961} showed that a Lie group action with compact isotropy is linearizable at $\cl O$ if and only if if it locally proper at $x$.
    \end{example}

    In anology with the previous example, our main source of linearizable groupoids comes from proper Lie groupoids, by 
    
    \begin{theorem}
      \label{thm:8}
      A proper Lie groupoid $G \rra M$ is linearizable at every orbit. More generally, if $G \rra M$ is proper at $x$ (meaning $(s,t)$ is proper at $(x,x)$; see Definition 1.1 in \cite{CS2013}), it is linearizable at the orbit $\cl{O}$ through $x$. 
    \end{theorem}

    For this version of linearizability of proper Lie groupoids, we use Crainic and Struchiner's work \cite{CS2013}. As alluded in the introduction, earlier fundamental work in this direction was done by Weinstein \cite{W2002} and Zung \cite{Z2006}. Note that a Lie groupoid may be proper at each $x \in M$, hence linearizable at each orbit, without being globally proper.

Of interest to us is the fact every linearizable Lie groupoid induces a singular foliation that is decomposed by dimension. In the context of proper groupoids, this fact is also proved by Posthuma, Tang, and Wang in \cite{PTW2021}. The authors rely the linearization theorem for proper Lie groupoids given by \cite{CS2013}, as well as a description of the linear model provided in \cite{PPT}. Because linearizable Lie groupoids are a key example for us, and we prefer to avoid assuming properness, we give a self-contained - but not fundamentally dissimilar - reproduction of this fact.

\begin{proposition}
  \label{prop:7}
  A linearizable Lie groupoid induced singular foliation decomposed by dimension. 
\end{proposition}

\begin{proof}
  Fix $x \in M$, with orbit $\cl O$. By linearizability, get open $U$ and $V$ about $\cl O$, and an isomorphism of the Lie groupoids $\cl G|_U \rra U$ and $(\cl G_{\cl O} \ltimes \nu(\cl O))|_V \rra V$. Since the dimension of the orbit about $y \in U$ is entirely determined by $\dim G_y$, we relate $G_y$ to $G_x$.

Say $y$ corresponds to $v$. Viewing $\nu(\cl O)$ as $P_x \times_{G_x} \nu_x \cl O$, we may find unique $[k, w]$ in the associated bundle such that $k\cdot w = v$. We propose the following diffeomorphism:

  \begin{equation*}
    (G_{\cl O} \ltimes \nu(\cl O)) \supseteq \text{iso}(v) \to k \stab_{G_x}(w)k^{-1} \subseteq \cl G_{\cl O}, \quad (g, v) \mapsto g.
  \end{equation*}

  \begin{itemize}
  \item Well-defined: the only arrows from $v$ in the action groupoid are of the form $(g, v)$, because the source map is just the projection of the second coordinate. Also, by definition of the action,
    \begin{align*}
      g: \pi(v) &\mapsto \pi(g\cdot v) = \pi(v) \\
      k: x &\mapsto \pi(k \cdot w) = \pi(v), 
    \end{align*}
    hence $k^{-1}gk \in G_x$. Moreover
    \begin{equation*}
      k^{-1}gk \cdot w = k^{-1}g \cdot v = k^{-1} \cdot v = w.
    \end{equation*}
    So finally, we can write $g = k(k^{-1}gk)k^{-1} \in k \stab_{G_x}(w)k^{-1}$.

  \item Smooth: this is the restriction of a projection, and both domain and codomain are embedded submanifolds.
  \item Inverse: the inverse map $g \mapsto (g, v)$ is well-defined because $g = k\gamma k^{-1}$ for some $\gamma \in \stab_{G_x}(w)$, hence
    \begin{equation*}
      g \cdot v = k \gamma k^{-1} \cdot v = k \gamma \cdot w = k \cdot w = v.
    \end{equation*}
    It is smooth since it is the restriction of an inclusion.
  \end{itemize}

  Therefore, $G_y \cong \text{iso}(v) \cong k\stab_{G_x}(w)k^{-1}$. In particular, $\dim G_y = \dim \stab_{G_x}(w)$, and  $\dim G_y = \dim G_x$ if and only if $\stab_{G_x}(w)$ is an open submanifold of $G_x$ about the identity. In this case, necessarily $\stab_{G_x}(w) \supseteq G_x^\circ$, where $G_x^\circ$ is the identity component of $G_x$. In other words, the set of $w$ fixed by $G_x^\circ$ corresponds exactly to those $y$ such that $\dim G_y = \dim G_x$. Denote by $(\nu_x\cl O)^{G_x^\circ}$ the vector subspace of fixed points of $G_x^\circ$. In this notation, we conclude the diffeomorphism $U \to V$ from the linearization descends to a bijection $M_{=k} \cap U \to V \cap P_x \tensor[]{\times}{_{G_x}} (\nu_x \cl O)^{G_x^\circ}$. The codomain is an embedded submanifold of $V$, and thus we take these bijections as charts for an atlas of $M_{=k}$. These make $M_{=k}$ into an embedded submanifold of $M$. The components of $M_{=k}$ may have different dimensions.
\end{proof}

   \begin{example}
     A consequence of the proof above is that if a Lie groupoid is linearizable at $x$, there is a neighbourhood of the orbit through $x$ such that all isotropy groups are conjugate to a subgroup of $G_x$. We can use this to provide a linear action which is not linearizable. Consider the representation of $\text{SL}_2(\CC)$ on binary forms of degree $3$ (i.e.\ polynomials of the form $p(x,y) = \sum a_i x^{3-i} y^i$), given by $M \cdot p(x,y) := p(M(x,y)^T)$. This is the unique irreducible representation of $\text{SL}_2(\CC)$ in 4 (complex) dimensions. The isotropy of $x^2y$ is trivial, yet in any neighbourhood of $x^2y$ there is a form with isotropy of order at least $3$; this is nontrivial, see \cite{P1994} page 162. Therefore the action groupoid is not linearizable around the orbit of $x^2y$.
   \end{example}
   
\section{Pseudogroups}
\label{sec:pseudogroups}

  A pseudogroup is a group of locally-defined diffeomorphisms, and these arise naturally in the context of singular foliations.  For this discussion, fix a smooth manifold $M$. If $f:U \to U'$ and $g:V \to V'$ are two smooth functions on open subsets of $M$, denote by $g \circ f$ the restricted composition $g\circ f:f^{-1}(V) \to g(U')$; note we restrict the codomain. As a special case, for a subset $W$ of $\dom f$, denote by $f|_W$ the restriction $f|_W:W \to f(W)$.  We allow for compositions to be the empty map.  Finally, given a set of maps $A$, let $A^{-1} := \{f \mid f^{-1} \in A\}$.

  \begin{definition}
    \label{def:23}
    A \emph{transition} on $M$ is a diffeomorphism $f:U \to U'$ of open subsets of $M$. A \emph{pseudogroup} $P$ is a collection of transitions such that
    \begin{enumerater}
    \item $\id|_U$ is in $P$ for every open subset $U$.
    \item If $f$ and $f'$ are in $P$, then so are $f' \circ f$ and $f^{-1}$.
    \item If $f:U \to U'$ is a transition and $\{U_i\}_{i \in I}$ is a cover of $U$, and we have $f|_{U_i}$ in $P$ for each $i \in I$, then $f$ is in $P$.
    \end{enumerater}
  \end{definition}

  The intersection of an arbitrary collection of pseudogroups is itself a pseudogroup. In particular, given an arbitrary set of transitions $A$, the \emph{pseudogroup generated by }$A$ is defined as
  \begin{equation*}
    P_A := \bigcap_{A \subseteq P'} P',
  \end{equation*}
  where $P'$ runs over all pseudogroups containing $A$. It is the minimal pseudogroup containing $A$. If $P$ is a pseudogroup and $A$ is a set of transitions with $P_A = P$, we say $P$ is generated by $A$. If $A$ can be chosen countable (finite), we say $P$ is countably (finitely) generated. A pseudogroup generated by $A$ consists exactly of those transitions that are locally compositions of elements of $A \cup A^{-1}$, in the following sense.
  \begin{lemma}
    \label{lem:5}
    Suppose $P$ is a pseudogroup generated by a set of transitions $A$. Then $f:U \to U'$ is in $P$ if and only if about each $p \in U$, there is an open neighbourhood $V$ in $U$ and transitions $f_1, \ldots, f_n \in A \cup A^{-1}$ such that
    \begin{equation*}
      f|_V = (f_1\circ \cdots \circ f_n)|_V.
    \end{equation*}
  \end{lemma}
  The proof is straightforward.

  \subsection{The germ groupoid}
\label{sec:germ-groupoid}

Here assume $M$ is Hausdorff and second countable. Associated to each pseudogroup $P$ on $M$ is the \emph{germ groupoid} $\Gamma P$, a Lie groupoid. In the sequel we denote $\Gamma P$ by $\Gamma$ for simplicity. Its base manifold is $\Gamma_0 = M$. Its set of arrows $\Gamma$ consists of all germs of transitions in $P$, that is
  \begin{equation*}
    \Gamma := \{\germ_x f \mid f \in P, \ x \in \dom f\}.
  \end{equation*}

  The source of $\germ_x f$ is $x$, and the target is $f(x)$, so that $\germ_xf:x \mapsto f(x)$.  The multiplication is given by the composition of germs: if $f,g \in P$ and $f(x) \in \dom g$, then $\germ_{f(x)} g \cdot \germ_x f := \germ_x (f \circ g)$. The unit $x \mapsto x$ is the germ of the identity, and the inverse of $\germ_xf$ is $\germ_{f(x)} f^{-1}$. With these structures, it is not hard to see $\Gamma$ is a groupoid.

  To see it is Lie, we require a smooth structure on $\Gamma$ such that the source map $s$ is a submersion with Hausdorff fibers, and the other structure maps are smooth.  Given $(f:U \to U') \in P$, consider the map
  \begin{equation*}
   U \to \Gamma, \quad x \mapsto \germ_x f,
  \end{equation*}
  and denote its restriction to its image by $\tilde f$.  Let $\Psi := \{\tilde f \mid f \in P\}$.  We claim $\Psi$ is an atlas (with values in $M$) for $\Gamma$ making $\Gamma \rra M$ an \'{e}tale Lie groupoid.

Recall in Section \ref{sec:class-lie-group} we associated to each Lie groupoid $\cl G \rra M$ a pseudogroup $\Psi(\cl G)$ on $M$, whose elements were $t \circ \sigma$ for all local bisections $\sigma$.  We have $\Psi(\Gamma(P)) = P$.  As for $\Gamma(\Psi(G))$, there is a natural morphism $\Eff: G \to \Gamma(\Psi(G))$ which is the identity on the base, and maps the arrow $g$ to $\Eff(g) :=  \germ_{s(g)} (t \circ \sigma)$, where $\sigma$ is a section of $s$ through $g$ (whose germ is unique because $s$ is \'{e}tale). This functor (on arrows) is onto, but is not injective in general.

  \begin{definition}\label{def:24}
    We call $\Gamma(\Psi(G))$ the \emph{effect} of $G$, and denote it by $\Eff(G)$.  An \'{e}tale groupoid $G$ is \emph{effective} if the functor $\Eff:G \to \Gamma(\Psi(G))$ is injective, and hence an isomorphism of Lie groupoids.
  \end{definition}

\end{document}